\documentclass[11pt,oneside,reqno]{amsart}
\usepackage[utf8]{inputenc}
\pdfoutput=1
\usepackage{parskip}
\usepackage{mathrsfs}
\usepackage[margin=1in]{geometry}
\usepackage{enumerate}
\usepackage{amsthm}
\usepackage{amsmath}
\usepackage{amsfonts}
\usepackage{amssymb}
\usepackage{graphicx}
\usepackage{color}
\usepackage{graphics}
\usepackage{eepic}
\usepackage{nicefrac}
\usepackage{bigints}
\usepackage{bbm}
\usepackage{array}
\usepackage{caption}

\makeatletter
\@addtoreset{equation}{chapter}
\makeatother
\usepackage{tikzsymbols}
\usepackage{booktabs}
\usepackage{hyperref}
\hypersetup{
 colorlinks   = true,
 urlcolor     = blue,
 linkcolor    = blue,
 citecolor   = red,
 bookmarksopen=true
}

\newcommand{\ignore}[1]{}

\usepackage[colorinlistoftodos]{todonotes}

\newcommand{\be}{\begin{equation}}
\newcommand{\ee}{\end{equation}}

\renewcommand{\Re}{\operatorname{Re}}

\newcommand{\supp}{\operatorname{supp}}

\newcommand{\C}{{\mathbb{C}}}

\newcommand{\N}{{\mathbb{N}}}

\newcommand{\D}{{\mathbb{D}}}

\newcommand{\PP}{{\mathbb{P}}}

\usepackage{mathtools} 

\newcommand{\diam}{\mbox{diam}}

\newtheorem{thm}{Theorem}[section]

\newtheorem{prop}[thm]{Proposition}

\newtheorem{Lemma}[thm]{Lemma}

\newtheorem*{claim*}{Claim}
\newtheorem{question}[thm]{Question}

\theoremstyle{definition}
\newtheorem{defn}[thm]{Definition}
\newtheorem{example}[thm]{Example}

\theoremstyle{remark}

 \renewcommand\epsilon{\varepsilon}

\usepackage{palatino}
 \usepackage{color, xcolor, mathrsfs, graphicx, hyperref, framed}

\newcommand{\mb}{\mathbb}

\usepackage[italian,english]{babel}
\usepackage[autostyle]{csquotes}  

\title[On the number of components of polynomial lemniscates]{Number of components of polynomial lemniscates: a problem of Erd\"os, Herzog, and Piranian.}

\author{Subhajit Ghosh}
\address{Tata Institute of Fundamental Research, Centre for Applicable Mathematics, Bangalore
560065, India}
\email{subhajitg@tifrbng.res.in}

\author{Koushik Ramachandran}
\address{Tata Institute of Fundamental Research, Centre for Applicable Mathematics, Bangalore, India-560065}
\email{koushik@tifrbng.res.in}

\begin{document}

\begin{abstract}
Let $K\subset\C$ be a compact set in the plane whose logarithmic capacity $c(K)$ is strictly positive. Let $\mathscr{P}_n(K)$ be the space of monic polynomials of degree $n,$ \emph{all} of whose zeros lie in $K.$ For $p\in \mathscr{P}_n(K),$ its filled \emph{unit leminscate} is defined by $\Lambda_p = \{z: |p(z)| < 1\}.$ Let $\mathcal{C}(\Lambda_p) $ denote the number of connected components of the open set $\Lambda_p,$ and define $\mathscr{C}_n(K) = \max_{p\in \mathscr{P}_n(K)}\mathcal{C}(\Lambda_p).$ In this paper we show that the quantity 
\[M(K) = \limsup_{n\to\infty}\dfrac{\mathscr{C}_n(K)}{n},\]
satisfies $M(K) < 1$ when the logarithmic capacity $c(K) < 1,$ and  $M(K) = 1$ when $c(K)\geq 1.$ In particular, this answers a question of Erd\"os et. al. posed in $1958$. In addition, we show that for nice enough compact sets whose capacity is strictly bigger than $\frac{1}{2}$, the quantity $m(K) = \liminf_{n\to\infty}\dfrac{\mathscr{C}_n(K)}{n} > 0.$
\end{abstract}

\maketitle

\section{Introduction}\label{sec1}


Let $p(z)$ be a complex polynomial. The filled \emph{unit lemniscate} of $p$ is defined by

\[\Lambda_p = \{z\in\C: |p(z)| < 1\}.\]

\vspace{0.1in}

\noindent The study of lemniscates has a long and rich history with applications to complex approximation theory, potential theory, and fluid dynamics. We refer the interested reader to the surveys and references listed in \cite{lemofalgBF, TopolofalgvarKKP, polylemtreesbraidsFM, KLR} for more more on this fascinating subject. Of particular relevance to us is the influential paper of Erd\"os, Herzog, and Piranian, c.f. \cite{EHP}, where the authors initiated the study of various metric and topological properties associated to $\Lambda_p$. These include the area of $\Lambda_p,$ the number of connected components of $\Lambda_p,$ and also the length of the boundary $\partial\Lambda_p$.  Pommerenke, in \cite{metricpropertiespommerekee}, continued this research and settled several open problems listed in \cite{EHP}. In this paper, we focus on the number of connected components of lemniscates.

\vspace{0.1in}

It is easy to see that $\Lambda_p$ is a non-empty bounded open set. An application of the maximum modulus principle shows that each connected component of $\Lambda_p$ must contain at least one zero of $p$. It follows that $\Lambda_p$ has at most $\deg(p)$ many connected components. If we denote $\mathcal{C}(\Lambda_p) $ to be the number of connected components of the lemniscate $\Lambda_p,$ the preceding argument yields  the estimate $1\leq \mathcal{C}(\Lambda_p)\leq \deg(p)$. Let $m\in\N$ be given. For a polynomial of the form $p(z)= (z-a)^m$ we clearly have $\mathcal{C}(\Lambda_p) = 1.$ On the other hand, let us consider $q(z) = \prod_{j=1}^{m}(z-a_j),$ where $\min_{i\neq j\\ 1\leq i, j\leq n}|a_i - a_j|\geq d > 0$.  Then if $d\gg 1$ is sufficiently large, $\Lambda_q$ will have $m$ components. Loosely speaking, the reason is as follows: if $z$ is near a zero $a_j,$ then $z$ will be very far away from every other $a_i$, $i\neq j,$ making the product of the remaining distances $\prod_{i\neq j}|z-a_i|$  to be very large. Hence, for $|q(z)| < 1$ to hold, $z$ needs to be very close to each of the zeros of $q$. This forces $m$ isolated small components for $\Lambda_q$. By modifying this procedure, it is not hard to see that given an integer $m\in\N,$ and $k\in [m],$ we can find a monic polynomial $p$ of degree $m$ with $\mathcal{C}(\Lambda_p) = k.$ All of this shows that if the zeros of the polynomial are allowed to be arbitrarily far apart, the question of the number of components is well understood. 

\vspace{0.1in}

The above discussion leads to a natural question, originally posed by Erd\"os, Herzog, and Piranian in \cite{EHP}. If we constrain all the zeros of $p$ to lie on a compact set $K$ so that they cannot be arbitrarily far apart,  how large can $\mathcal{C}(\Lambda_p)$ be? Under the same constraint, what geometric features of $K$ determine the maximal number of components? 

\vspace{0.1in}

\noindent We now put things into a more formal framework. Let $K\subset\C$ be a compact set in the plane.  For $n\in\N,$ let $\mathscr{P}_n(K)$ denote the space of \emph{monic} polynomials of degree $n$, \emph{all} of whose zeros lie in the compact $K,$ i.e.,
\[\mathscr{P}_n(K) = \{p(z): p(z) = \prod_{j=1}^{n}(z - z_j), \hspace{0.05in}\mbox{where}\hspace{0.05in} z_j\in K \hspace{0.05in}\mbox{for}\hspace{0.05in} 1\leq j\leq n\}\]

Let $\mathscr{C}_n(K) = \max_{p\in\mathscr{P}_n(K)}\mathcal{C}(\Lambda_p)$. Thus $\mathscr{C}_n(K)$ is the maximal number of components a lemniscate can have from the class $\mathscr{P}_n(K).$ Define the quantities $M(K)$ and $m(K)$ by 
\begin{equation}\label{maxmin}
M(K) = \limsup_{n\to\infty}\dfrac{\mathscr{C}_n(K)}{n}, \hspace{0.1in}\mbox{and}\hspace{0.1in} m(K) = \liminf_{n\to\infty}\dfrac{\mathscr{C}_n(K)}{n}
\end{equation}
We can then ask how $M(K)$ and $m(K)$ behave for arbitrary compact sets $K$. We show below that this behavior depends on the logarithmic capacity of $K$. For the benefit of the reader, we recall the definition of capacity and other basic notions from potential theory. For a beautiful introduction to potential theory in the plane and its applications to the theory of polynomials, we refer the reader to \cite{Ransfordpotentialthryincomplexplane, SaffTotikLogpotential}.

\begin{defn}\label{cap}
Let $K\subset\C$ be a non-empty compact set, and $\mu$ a Borel probability measure on $K.$ 

\begin{enumerate}\label{potential}
\item The logarithmic potential of $\mu$ is the function $U_{\mu}:\C\rightarrow [-\infty, \infty)$ defined by 

\[U_{\mu}(z) = \int_{K}\log|z-w|d\mu(w), \hspace{0.05in}z\in\C.\]

\vspace{0.1in}

\item The energy associated with the measure $\mu$ is defined by
\[I(\mu) = \int_{K}\int_{K}\log|z-w|d\mu(z) d\mu(w)= \int_{K}U_{\mu}(w)d\mu(w).\] 

\vspace{0.1in}

\item Finally, the logarithmic capacity of $K$ is defined by
\[\log c(K) = \sup_{\mu\in\mathcal{P}(K)}I(\mu),\]
where $\mathcal{P}(K)$ denotes the collection of all Borel probability measures on $K.$
\end{enumerate} 
\end{defn}

If $c(K) > 0,$  it is known that there exists a unique probability measure $\nu = \nu_K$ on $K$,  which satisfies

\[I(\nu) = \sup_{\mu\in\mathcal{P}(K)}I(\mu).\]

The measure $\nu$ is called the equilibrium measure of the compact $K.$ The capacity can then be expressed as $\log c(K) = I(\nu).$ 

\vspace{0.in}

\noindent After this brief digression, let us go back to our problem. Erd\"os, Herzog, and Piranian in \cite[problem $6$]{EHP} ask the following
\begin{question}\label{EHPQ} Is it true that $M(K) < 1$ when the logarithmic capacity $c(K) < 1$?  Furthermore, if $c(K) = 1$ and $K$ is not contained in a closed disk of radius $1,$ can $\mathscr{C}_n(K)$ ever be $n$ (up to a subsequence)? 
\end{question}

\vspace{0.1in}

\noindent In this paper we answer Question \ref{EHPQ} in the affirmative. In addition, we show that for compacts $K$ with $c(K) > 1,$ we have $M(K)= m(K) =1,$ while $m(K) > 0$ if $c(K) > \frac{1}{2}$. To explain why the condition in the second part of Question \ref{EHPQ} is imposed, consider for each $n$, the polynomials $p_n(z) = z^n - 1.$ All the zeros of $p_n$ are contained in the closed unit disk and it is not hard to see that $\Lambda_{p_n}$ has $n$ connected components, see Example \ref{Erdosp} below.

\vspace{0.1in}

\noindent We are now ready to state our results but before that, we remind the reader of the notation we will use for the rest of this paper. 

\subsection{Notation}
$K \subset \mb{C}$ will always denote a compact set with positive logarithmic capacity.  
\begin{itemize}
    \item $\mathscr{P}_n(K) := \left\{ p(z) = \prod_{j=1}^n(z-z_{j}) :  z_{j} \in K , \hspace{0.05in}1\leq j\leq n\right\}.$

    \item $\Lambda_p:= \{z\in\C: |p(z)| < 1\}.$

    \item $\mathcal{C}(\Lambda_p):= $ the number of connected components of $\Lambda_p.$

    \item $\mathscr{C}_n(K) := \max \left\{ \mathcal{C}\left(\Lambda_p\right) : p \in \mathscr{P}_n(K) \right\}$


    \item 
$M(K) = \limsup_{n\to\infty}\dfrac{\mathscr{C}_n(K)}{n},$ and  $m(K) = \liminf_{n\to\infty}\dfrac{\mathscr{C}_n(K)}{n}$

    \item $\mathcal{P}(K) := \text{the collection of all Borel probability measures on $K.$}$

    \item $c(K) := \text{the logarithmic capacity of } K$

    \item $B(a,r):= \{z \in \mb{C}: |z-a|<r \}$ 
\end{itemize}
Throughout the paper, we denote by $C$ a positive numerical constant whose value may vary from line to line.

\section{Main results}

\begin{thm}\label{capless1}
   Suppose that $0< c(K) < 1$.

    \begin{enumerate}[(a)]
   \item Then $M(K) < 1.$
   
\vspace{0.05in}

   \item Furthermore, if $c(K)\in (\frac{1}{2}, 1)$, and $K$ is either the closure of a bounded Jordan domain or a $C^2$ Jordan arc, then $m(K) > 0.$

\vspace{0.05in}

  \item If $c(K) \leq \frac{1}{4}$ and $K$ is connected, then $\mathscr{C}_n(K) = 1$ for all $n$. Consequently $m(K) = 0.$
    \end{enumerate}

     
\end{thm}

\noindent Theorem \ref{capless1} $(b)$ is sharp in the sense that there exists a compact $K$ of capacity $\frac{1}{2}$ for which $m(K)=0$. Indeed, take the compact set $K$ to be $\overline {B(0, \frac{1}{2})}$. It is well known that this has capacity $\frac{1}{2}$. For any polynomial $p \in \mathscr{P}_n(K)$, we can write 
    \begin{align*}
        p(z) = \prod_1^n(z-a_i), \quad \textit{where} \quad |a_i| \leq \frac{1}{2}, \quad \forall i=1,\cdots,n. 
    \end{align*}
Then for all $z \in B(0,\frac{1}{2}),$ we have $|z-a_i|<1$ by the triangle inequality. Therefore, $\overline{B(0,\frac{1}{2})} \subset \Lambda_{p}$. Since all the zeros of $p_n$ lie in the connected set $\overline{B(0,\frac{1}{2})}$ this yields that $\mathcal{C}(\Lambda_{p})=1$, and hence $m(K) = 0.$ 

\vspace{0.1in}

\noindent We remark that in general $m(K)$ and $ M(K)$  do not depend \emph{only} on the capacity of $K.$ For instance $K_1 = \overline{B(0, \frac{1}{2})}$ and $K_2= [-1, 1]$ both have capacity $\frac{1}{2}.$ However $M(K_1) = 0$ as shown above, but $M(K_2)>0$ as can be shown using the same ideas as in the proof of Theorem \ref{capless1} $(b).$ It would be interesting to know what other geometric features of $K$ determine $M(K)$ and $ m(K).$

\vspace{0.1in}
    
Theorem \ref{capless1} $(c)$ is also sharp in the sense that connectedness is essential. In general, one can construct a disconnected set $K$ of arbitrary small capacity so that $m(K)$ is positive. For instance, consider two line segments $B_1:=[n,n+\epsilon]$ and $B_2:=[-n-\epsilon,-n]$, each of length $\epsilon$ which are distance $2n$ apart from each other. Then the capacity of $B:=B_1 \cup B_2$ is  (c.f. \cite[Corollary $5.2.6$]{Ransfordpotentialthryincomplexplane})
    \begin{align*}
         c(B) = \frac{1}{2}\sqrt{(2n+\epsilon) \epsilon}.    
    \end{align*}
    which can be made arbitrarily small by appropriate choice of $n$ and $\epsilon.$ Now using the same ideas as in the proof of Theorem \ref{capless1} part $(b)$,  one can show that $m(B)$ is positive.



\begin{thm}\label{capbigg1}
Let $c(K) > 1.$ Suppose in addition that at least one of the following conditions hold
\begin{enumerate}[(a)]
   \item  The equilibrium measure $\nu$ of $K$ satisfies
  \begin{align}\label{regularity}
        \nu (B(z, r)) \leq  Cr^{\epsilon},  
    \end{align}
\noindent for all $r\in (0, \infty),$ and some constants $C, \epsilon > 0.$

    \item $K$ is connected.
\end{enumerate}

\vspace{0.1in}

Then $\mathscr{C}_n(K) = n$ for all large $n.$ In particular, $M(K) = m(K) = 1.$
\end{thm}

It is worthwhile to remark here that a large class of compact sets is covered by the two conditions $(a)$ and $(b)$ of Theorem \ref{capbigg1}. Indeed, if for instance, $K$ is the disjoint union of finitely many bounded $C^2$ domains (along with their closures), it is well known that the equilibrium measure satisfies \eqref{regularity}. On the other hand, part $(b)$ imposes no restrictions on the smoothness of $K$ but a mild topological one in the form of connectedness. There are generalized Cantor sets of arbitrarily large capacity that are totally disconnected and fractal-like,  for which Theorem \ref{capbigg1} does not apply. We believe the same result holds for such sets but have been unable to prove it.

Let $K$ be a compact set of capacity $1$. Following the definition in \cite{asymptoticofchebyshev}, we call $K$ a \emph{period-$m$ set} if there is a monic polynomial $P_m$ of degree $m$ such that, 
    \begin{align} \label{period m set}
        K = P_m^{-1}\left([-2,2]\right).
    \end{align}
The polynomial $P_m$ is called the \emph{generating polynomial} of the \emph{period-$m$ set} $K$. Similarly, we call a compact set $K$ a \emph{ closed lemniscate}, if there is a monic polynomial  $Q$, such that
			\begin{align}\label{generating poly}
			    K=Q^{-1}\left( \overline{B(0,1)} \right).
			\end{align}

Note that by the definition a \emph{closed lemniscate} has capacity $c(K) = 1.$ We point out that the polynomial $Q$ is not unique. In particular, any power of $Q$ also satisfies \eqref{generating poly}. Among all polynomials satisfying \eqref{generating poly}, the one with the lowest degree is called the \emph{generating polynomial} of $K$. It can be shown that the \emph{generating polynomial} is unique by considering the Green's function of the unbounded component of $\widehat{\C}\setminus{K}$ with pole at $\infty.$ We leave the details to the interested reader. The unique \emph{generating polynomial} for a \emph{lemniscate} will be denoted by $Q_m.$




\begin{prop}\label{proposition}
      If $K$ is either a \emph{closed lemniscate} or a period-m set, then we have $M(K)=1$. Moreover, for a lemniscate $K$, there exists a subsequence $\mathcal{N} \subset \N$ such that 
     \begin{align*}
         \mathscr{C}_n(K)=n , \quad \forall n \in \mathcal{N}.
     \end{align*}
\end{prop}

\begin{thm}\label{capeq1M}
Let $c(K) = 1.$ Suppose in addition $K$ is the closure of a bounded Jordan domain with $C^2$ boundary. Then, 
\begin{align}
    \frac{1}{2} \leq m(K) \leq M(K) =1.
\end{align}
\end{thm}

\vspace{0.1in}

The following lemma which relates the number of connected components of a lemniscate to the critical values of the polynomial will be repeatedly used in the proofs. We collect it here.

\begin{Lemma}\label{Components and critical points}
        Let $p$, $\Lambda_p$, and $\mathcal{C}(\Lambda_p)$ have the same meaning as in Section \ref{sec1}. Let $\{\beta_j\}_{j=1}^{n-1}$ be the set of critical points of $p$. Then,
        \begin{align*}
           \mathcal{C} \left(\Lambda_p\right)=1+ \left|\left\{\beta_j: |p(\beta_j)| \geq 1\right \}\right|.
        \end{align*}    
\end{Lemma}

A proof of this lemma can be found in \cite{ghosh2023number}, or in \cite[Proposition $2.1$]{Twodimshapeskhavinson}. Next, we give two classical examples where the exact behaviour of $\mathscr{C}_n(K)$ can be determined.

\begin{example}\label{Erdosp}{\textbf{(Closed unit disk $\overline{\mb{D}}$)}}\\
     Consider the sequence of polynomials $P_n(z)=z^n-1$, whose zeros are the $n^{\text{th}}$ roots of unity lying on the unit circle. The critical points of $P_n$ are at zero, with multiplicity $(n-1)$, and the critical values are of absolute value $|P_n(0)| = 1$. Therefore, by Lemma \ref{Components and critical points}, we see that $\Lambda_{P_n}$ has exactly $n$ components. 
\end{example}

\begin{example}{\textbf{(Chebyshev polynomial in $[-2,2]$})}\\
    Let $T_n(z)$ be the Chebyshev polynomials of degree $n$ for the compact set $[-1,1]$, then,
    \begin{align}\label{a1}
        T_n(x) = 2^{n-1}x^n+.... = \cos(n\theta), \quad \textit{ where } x = \cos(\theta). 
    \end{align}
    We know that the zeros and critical points of $T_n(z)$ are respectively, \big(c.f. \cite{chebyshevpolynomials}\big)
    \begin{align}\label{a2}
       x_k = \frac{\cos(k - \frac{1}{2}) \pi }{n},\quad  k = 1,2,\cdots,n. \quad \quad  \beta_k = \cos \left(\frac{k\pi}{n}\right), \quad k = 1,2,\cdots,n-1.
    \end{align}
    Since we are interested in monic polynomials, let us define the monic Chebyshev polynomials $\mathcal{T}_n(x):= \frac{1}{2^{n-1}} T_n(x)$. From \eqref{a1} and \eqref{a2} we can  compute the critical values  
    \begin{align}\label{a3}
       |\mathcal{T}_n(\beta_k)|=\frac{1}{2^{n-1}}|T_n(\beta_k)|= \frac{1}{2^{n-1}}|\cos(k\pi)|=\frac{1}{2^{n-1}}, \quad  k = 1,\cdots,n-1.
      \end{align}
      Now, one can define the monic Chebyshev polynomials for $[-2,2]$ just by scaling the zeros by a factor of $2$. That is, $\mathcal{T}_n^{[-2,2]}(x):= 2^nT_n({x}/{2}),$ where the factor of $2$ is multiplied to make $\mathcal{T}_n^{[-2,2]}$ monic. The zeroes and the critical points of $\mathcal{T}_n^{[-2,2]}$ are $2x_k$ and $2\beta_k$ respectively. Therefore, the critical values are 
      \begin{align}\label{a4}
          \left|\mathcal{T}_n^{[-2,2]}(2\beta_k)\right|=2^n\left|T_n\left(\frac{2\beta_k}{2}\right)\right|=2.
      \end{align}
     Using lemma \ref{Components and critical points} with \eqref{a4}, we get that the lemniscate for Chebyshev polynomials in $[-2,2]$ has $n$ components (see Figure-\ref{Chebyshev}).
\end{example}



 \begin{figure}
\centering
\includegraphics[scale=.775]{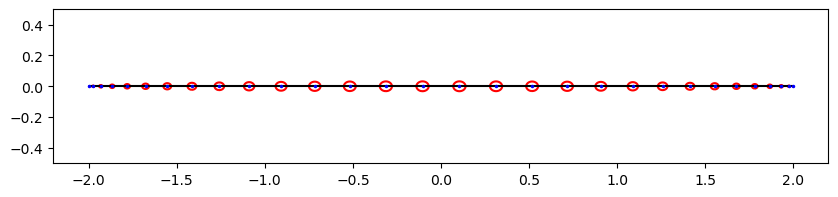}
\caption{Unit lemniscates of the Chebyshev polynomial of degree $30$ in $[-2,2]$ is plotted in red, along with the zeros shown in black.}\label{Chebyshev}
\end{figure}

 \begin{figure}
\centering
\includegraphics[scale=.325]{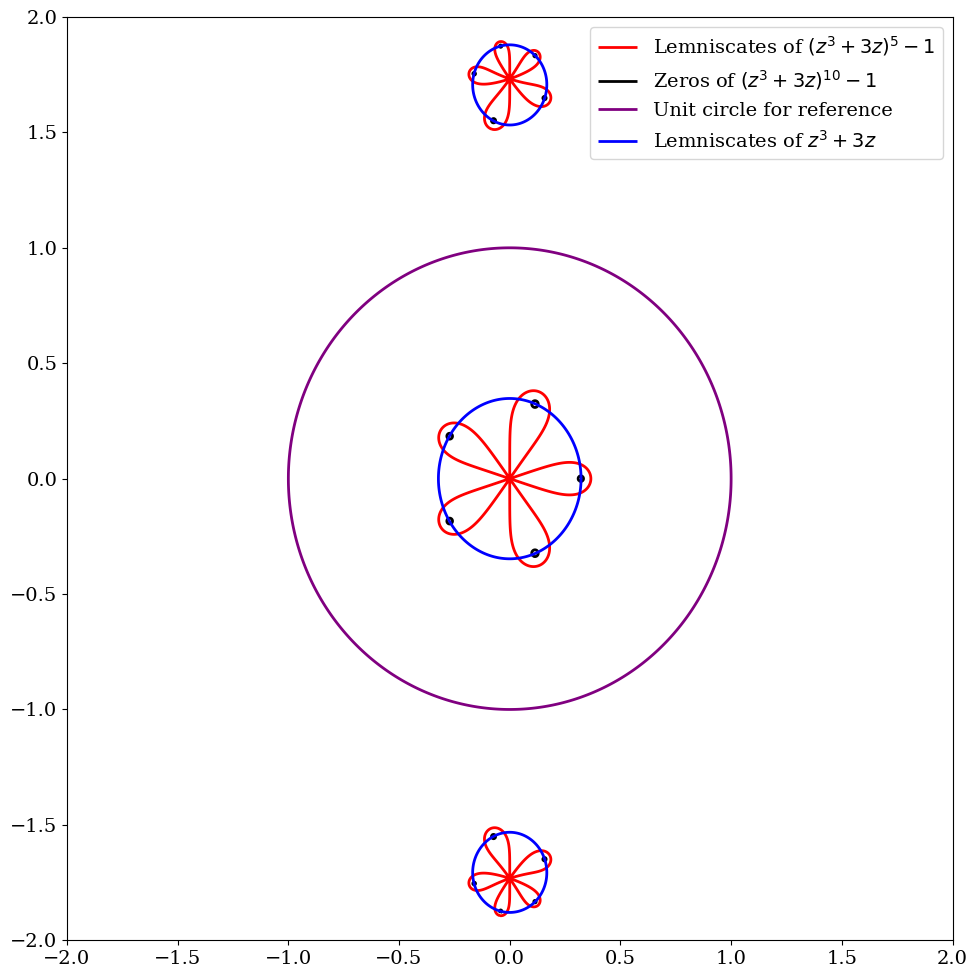}
\includegraphics[scale=.325]{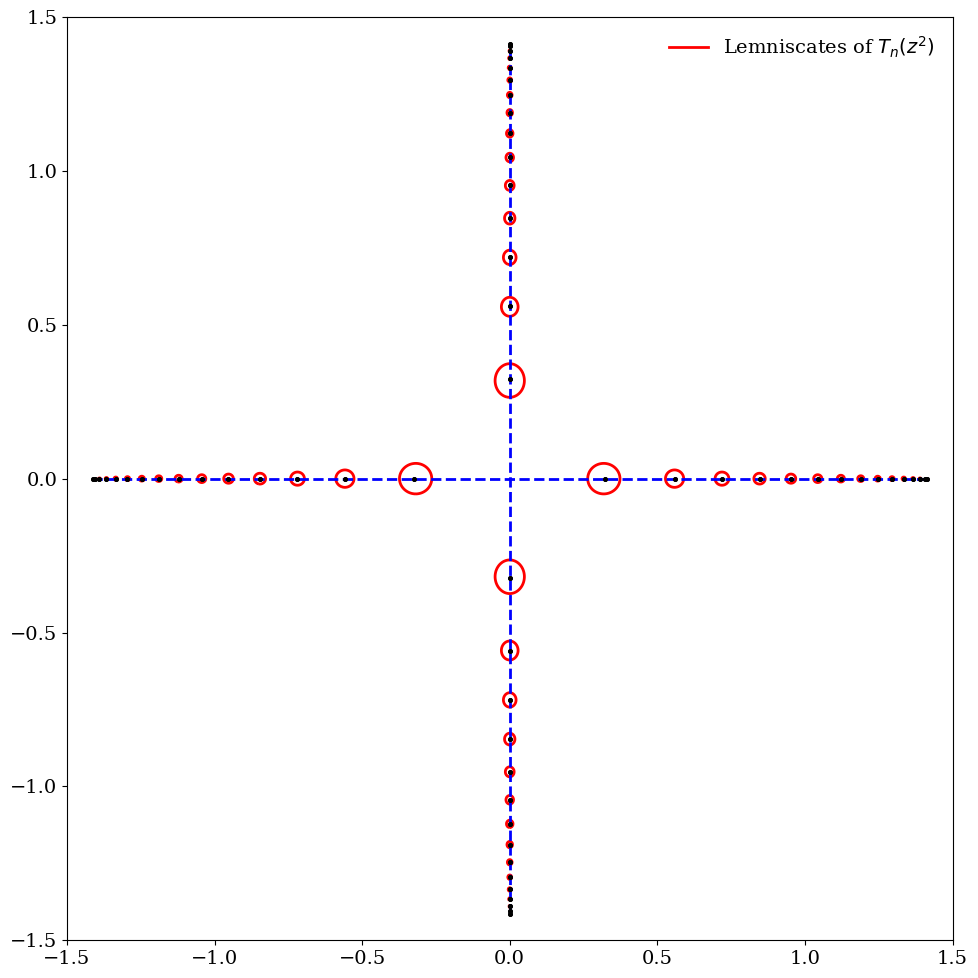}
\caption{In the left, the unit lemniscates of $(z^3+3z)^{5}-1$ are in red along with its zeros in black lying inside the compact set $(z^3+3z)^{-1} (B(0,1))$. On the right, the period-2 set $(z^2)^{-1}([-2,2])$ is plotted in blue and unit lemniscate of $\mathcal{T}_{30}(z^2)$ is plotted in red along with its zeros in black.}\label{Second example}

\end{figure}

 \subsection{Overview of the proofs}
In this section, we loosely sketch the ideas of the main theorems. Our proofs are based on tools from potential theory and probability. The connection between polynomials and potentials stems from the following well known fact: if $p$ is a polynomial with zero set $Z,$ and $\mu_p$ is the empirical measure of its zeros defined by
\[\mu_p:= \frac{1}{\deg(p)}\sum_{z\in Z}\delta_{z}\]

then,
\begin{equation}\label{polypotential}
\frac{1}{\deg(p)}\log|p(z)| = U_{\mu_p}(z).  
\end{equation}

 In the context of studying lemniscates, this relation implies that $\Lambda_p = \{U_{\mu_p}(z) < 0\}.$ We now proceed to explain the main ideas of how $c(K)$ determines $M(K)$ and $m(K).$

\subsubsection{\bf{Bounds on $M(K)$ and $m(K)$:}} 

Let $K$ be the compact in question. For each $n,$ pick a  $p_n\in\mathscr{P}_n(K)$ having $\mathscr{C}_n(K)$ connected components for its lemniscate. Let $\mu_n$ denote the empirical measure of the zeros of $p_n.$ Since all the measures $\mu_n$ are supported in $K$, some subsequence, which we continue to call $\mu_n$, will converge weakly to a probability measure $\mu$ with $\supp(\mu)\subset K.$ This gives (in an appropriate sense) that $U_{\mu_n}(z)\approx U_{\mu}(z).$ Now the behavior of $U_{\mu}$ depends on the capacity of $K.$ 

\vspace{0.1in}

\noindent If $c(K) < 1,$ then \emph{every} probability measure $\sigma$ supported in $K$ satisfies 
    
    \begin{equation}\label{negative}
     \{U_{\sigma} < 0\} \hspace{0.05in}\mbox{ in some open ball B with} \hspace{0.05in} \sigma(B) > 0.   
    \end{equation}
    
    Applying \eqref{negative} to our weak limit $\mu$ from the previous paragraph, and remembering that the measures and their potentials are close, we get that for large $n,$ $U_{\mu_n} < 0$ on $B,$ and $\mu_n(B) > c_1/2.$ This fact combined with \eqref{polypotential} shows that for all large $n$, there is one component of $\Lambda_{p_n}$ which contains at least a positive proportion of the zeros. In other words, $M(K) < 1$.

\vspace{0.1in}

\noindent If $c(K) > 1,$ then \eqref{negative} no longer holds for every measure. In fact, the equilibrium measure $\nu$ of $K$ satisfies  $U_{\nu}(z) > \log c(K) > 0$ for all $z\in\C.$ This works in our favor. For if we then take suitable point masses $\mu_n$ supported on $K$ so that $\mu_n\to\nu$, then $U_{\mu_n}(z)\approx U_{\nu}(z)$ forces $\{U_{\mu_n} < 0\}$ to be rather small (it cannot be empty since at the point mass it is $-\infty$). A careful choice of the spacing between the point masses ensures $\{U_{\mu_n}< 0\} = \{\frac{1}{n}\log|p_n|< 0\}$ has $n$ components, much stronger than $M(K) = 1$.

\vspace{0.in}

\noindent If $c(K) = 1$, then the equilibirum measure $\nu$  satisfies 

    \begin{equation}
      U_{\nu}(z) = 0, z\in K, \hspace{0.05in}\mbox{and} \hspace{0.05in} U_{\nu}(z) > 0, z\in\C\setminus K. 
    \end{equation}
    
    This presents the most interesting and challenging case for it is right at the ``edge'' of the previous two scenarios. The reasoning given in the capacity less than one case tells us that if we want to look for close to $n$ lemniscate components for some sequence $p_n\in\mathscr{P}_n(K),$ then we better choose $p_n$ so that the corresponding empirical measures $\mu_n$ converge weakly  to $\nu$ (any measure other than $\nu$ has non-empty negative region for its potential). But any choice will not do because the lemniscates are exceedingly sensitive to small perturbations of the measures. For instance consider $\mu_n$ and $\sigma_n$, the empirical measures of the zeros of $p_n(z) = z^n - 1$ and $q_n(z) = z^n - \frac{1}{n}$ respectively. Then both $\mu_n$ and $\sigma_n$ are supported in $\overline{\D}$, and both converge weakly to $\frac{d\theta}{2\pi}=\nu_{\overline{\D}}$. However $\Lambda_{p_n}$ has $n$ connected components (see example \ref{Erdosp} ), whereas $\Lambda_{q_n}$ has only one component, as the reader can check by applying Lemma \ref{Components and critical points}. We get around this issue by first proving $M(K) = 1$ when $K$ itself is a lemniscate (along with its boundary). We then approximate a Jordan domain with analytic boundary by a lemniscate, and show that this approximation is good enough to preserve $M(K)=1.$ In general we conjecture that for arbitrary compacts $K$ with $c(K) = 1,$ we have $M(K)=1$ but the proof has been elusive to us. 

\vspace{0.1in}

 We next explain the idea behind obtaining non-trivial lower bounds for $m(K).$ Our approach is to consider \emph{random} polynomials $P_n(z) = \prod_{j=1}^{n}(z- X_j)$, where the $X_j$ are i.i.d. random variables whose law $\mu$ is supported in $K.$ Our first observation is that if $c(K) > \frac{1}{2}$, the diameter of $K$ is strictly bigger than $1$. Using this, we prove that there is a choice of a measure $\mu$ which is supported at opposite ends of a diameter, for which the expected number of components $\mathbb{E}(\mathcal{C}_n)$ is larger than $cn$ for some $c>0.$ From here, an application of the probabilistic method shows that for some realization of the $X_j$, polynomials $p_n$ exist for each $n,$ which satisfy $\Lambda_{p_n}$ has at least $cn$ many components.

\vspace{0.1in}

\section{Preliminary Theorems and Lemmas}
In this section, we collect some preliminary analytic and probabilistic results which will be used in the proofs of the main theorems. The first one is the following construction due to Erd\"os, Herzog and Piranian, of polynomials whose lemniscate closures have nearly maximal component count. 


Our next lemma shows that for smooth enough compacts, the equilibrium measure $\nu$ is very well behaved on small balls.





\begin{Lemma}\label{Bound on probability for equilibrium measures}
    Let $L$ be a $C^{2}$ Jordan arc and $\nu_L$ be the equilibrium measure on $L$. Then for $z \in L$, and small $r>0$,  there exist constants $C_1,C_2 > 0 $ such that
    \begin{align*}
        C_1r \leq \nu_L\big(B(z,r)\big) \leq  C_2\sqrt{r}.
    \end{align*}
   On the other hand, if $L$ is a $C^{2}$ Jordan curve, then $\nu_L$ is of the order of arc length measure on small balls. That is, for some $C_3, C_4 >0 $
   \begin{align*}
        C_3r \leq \nu_L\big(B(z,r)\big) \leq  C_4 r
    \end{align*}
\end{Lemma}

The proof of this lemma follows from \cite[lemma $2.1$]{toticeqilibriummeasure}.

\begin{prop}[\textbf{Scaling of zeros and corresponding critical points}]\label{Scaling of zeros and corresponding critical points}
    Let $p(z):= \prod_i^m(z-z_i)$. For $t>0,$ we define the $t$-scaling of $p$ by
    \begin{align}\label{scaled poly def}
        p_{t}(z):= t^m p\left(\frac{z}{t}\right) = \prod_i^m(z-tz_i).
    \end{align}
    Then $\beta_1,\cdots,\beta_{m-1}$ are critical points of $p$ iff $t\beta_1,\cdots,t\beta_{m-1}$ are critical points of $p_{t}$. Furthermore, the corresponding critical values are related by 
     \begin{align}\label{scaling of critical value}
       |p_{t}(t\beta_j)|=t^m|p(\beta_j)|.
    \end{align}
    
\end{prop}

\begin{proof}
    The proof follows directly from the definition of $p_{t}(z)$ and routine computations, and is therefore omitted.
\end{proof}

        

\begin{Lemma}[\textbf{Components of lemniscates touch at critical points}]\label{touching of lemniscates}Let $p$ be a complex polynomial. Suppose that $U_1$ and $U_2$ are the boundaries of two components of the $r-$lemniscate $\Lambda_{p,r}:=\{z:|p(z)|<r\}$ which touch each other at $z_0$. Then $ p^{'}(z_0) = 0$.
\end{Lemma}

\begin{proof}
     If possible, let $ p^{'}(z_0)\neq 0$, then by inverse function theorem, there is a neighbourhood $U_0, V_0$ such that $p: U_0 \longrightarrow V_0$ is a diffeomorphism, in particular homeomorphism. Clearly this cannot happen as $p\left(U_0\cap (U_1 \cup U_2)\right)$ is not homeomorphic to $\mb{D} \cap V_0.$
\end{proof}

\begin{Lemma}[\textbf{Sufficient condition for isolated components}]\label{suffcondn}
Let $p(z) = \prod_{j=1}^{n}(z- z_j)$ be a polynomial of degree $n.$ Suppose that there exist constants $\alpha, \beta > 0$  such that
\begin{enumerate}
    \item $|p'(z_j)|\geq\exp(n^{\alpha})$ for all $1\leq j\leq n,$ and
    \item $\min_{j\neq k}|z_j - z_k|\geq\dfrac{1}{n^{\beta}}.$
\end{enumerate}
Then if $n$ is large enough, $\Lambda_p$ has $n$ connected components each of which contains exactly one zero. 
\end{Lemma}

\begin{proof}
 Let $\Omega$ be the connected component of $\Lambda_p$ containing $z_1,$ and let $d = \diam(\Omega).$ Then by Bernstein's inequality, see \cite{PommDeriv}, we have
 \begin{equation}\label{bernst}
 |p'(z_1)|\leq C\dfrac{n^2}{d}\vert|p\vert|_{\overline{\Omega}}= C\dfrac{n^2}{d}
 \end{equation}
for some absolute constant $C>0.$ By the hypothesis on the size of $p'(z_1)$ we deduce that 
\[d\leq Cn^2\exp(-n^{\alpha}).\] On the other hand, the second condition says that for $j\neq 1,$  $|z_1- z_j|\geq\dfrac{1}{n^{\beta}}> Cn^2\exp(-n^{\alpha})$ if $n$ is large enough. This proves that $\Omega$ contains only $z_1$ and no other zeros, thus finishing the proof.  
\end{proof}

\begin{Lemma}\label{Main lemma}
    Let $\mu_n \in \mathcal{P}(K)$ be a sequence such that $\mu_n \overset{\omega^{*}}{\longrightarrow } \mu $. Then for any compact set $F$ we have,
    \begin{align*}
       \underset{n \to \infty}{\overline{\lim}}  \sup_{F} U_{\mu_n} \leq \sup_{F} U_{\mu}.
    \end{align*}
   
\end{Lemma}

\begin{proof}[\textbf{Proof of Lemma \ref{Main lemma}}]
    Since $ U_{\mu_n}$ is a subharmonic function on the compact $F,$ it is bounded above and the supremum is attained at some $z_n \in F.$ That is, 
    \begin{align}
        \sup_{F} U_{\mu_n} =  U_{\mu_n}(z_n), \quad  \forall n \in \mb{N}
    \end{align}
    Now $U_{\mu_n}(z_n)$ will have a subsequence that converges to the limsup. Along this subsequence, since $z_n$ lies in a compact set,  we can get a further subsequence and a point $z_0 \in F$ such that $ z_n \to z_0 $ and
    \begin{align*}
        \lim_{k \to \infty } U_{\mu_{n_k}}(z_{n_k}) = \underset{n \to \infty}{\overline{\lim}}  \sup_{F} U_{\mu}
    \end{align*}
    Next,  the principle of descent (c.f. \cite[Theorem-$6.8$]{SaffTotikLogpotential}) which states that for probability measures $\mu_n$, $n = 1, 2, ...$ having support in a fixed compact set and converging to some measure $\mu$ in the $\mbox{weak}^{*}$ topology, we have 
    \begin{align*}
        \underset{n \to \infty }{\overline{\lim}} U_{\mu_{n}}(w_{n}) \leq  U_{\mu}(w^{*}),
    \end{align*}
    for any sequence $w_n \to w^{*}$ in the complex plane. Taking $w_n=z_n$, we reach the desired conclusion.
\end{proof}

The rest of the lemmas in this section are all probabilistic. We start by recalling the classical Berry-Esseen result.

 \begin{thm}\textbf{(Berry-Esseen)}\label{BE}
        Let $X_1,X_2,...$ be i.i.d. random variables with $\mb{E}X_i=0,\mb{E}X_i^2=\sigma^2,$ and $\mb{E}|X_i|^3=\rho <\infty .$ If $F_n(x)$ is the distribution function of $\frac{(X_1+...+X_n)}{\sigma \sqrt{n}}$ and $\Phi(x)$ is the standard normal distribution, then
        \begin{align}
            |F_n(x)-\Phi(x)| \leq \frac{3\rho}{\sigma^3 \sqrt{n}}.
        \end{align}
    \end{thm}
    \noindent The proof of Theorem \ref{BE} can be found in \cite[Theorem $3.4.17$]{durrett2019probability} .

    \begin{Lemma}\label{log moment bound for uniform on cicle}
         Let $X$ be a random variable taking values in a compact set $K$ with law $\mu$. Let $d$ be the diameter of $K$. Assume that for all $z \in K, r\leq d$, there exist constants $ \epsilon_1,\epsilon_2,M_1,M_2 \in (0,\infty)$ such that $\mu$ satisfies 
         \begin{align}\label{lemma hypo}
             M_1r^{\epsilon_1}  \leq \mu(B(z,r))\leq M_2r^{\epsilon_2}.
         \end{align}
         Fix p, and define the function $F_p(z):= \mb{E}\Big[\big|\log|z-X|\big|^p\Big]:K \to \mb{R}$. Then, there exist constants $C_1,C_2$ depending on $d, p,\epsilon_1,\epsilon_2,M_1,M_2$, such that
            \begin{align}\label{log moment ineq}
                C_1 \leq \inf_{z\in K}F_p(z) \leq  \sup_{z\in K}F_p(z) \leq  C_2.
            \end{align}
    \end{Lemma}
    \begin{proof}
    
         We use the layer cake representation and write
                \begin{align*}
                        \mb{E}\left[\big|\log{|z-X|}\big|^p\right]&=p\int_0^\infty t^{p-1}\mb{P}\left(\big|\log{|z-X|}\big|>t\right) dt \\
                        &=p\int_0^{2d}t^{p-1}\mb{P}\left(\big|\log{|z-X|}\big|>t\right)dt +p\int_{2d}^\infty t^{p-1}\mb{P}\left(\big|\log{|z-X|}\big|>t\right)dt,
                \end{align*}
            In the first integral, we dominate the probability inside the integral by 1. Whereas in the second integral, notice that $(\log{|z-X|})^{+}<2d$, therefore, the probability is non zero when $\log{|z-X|}$ is negative. Taking this into account and using the upper bound in \eqref{lemma hypo} we get,
                \begin{align*}
                   \mb{E}\left[\big|\log{|z-X|}\big|^p\right] &\leq p \int_0^{2d} t^{p-1} d t + pM_2\int_{2d}^\infty t^{p-1}e^{-t\epsilon_2}dt\\[.35em]
                    &\leq p{(2d)}^{p+1}\left(1+C\left(\epsilon_2\right)M_2\right).
                \end{align*}
                   
            The lower bound in \eqref{log moment ineq} follows similarly using the left inequality in \eqref{lemma hypo} along with the layer cake representation.
    \end{proof}
    \begin{Lemma}(\textbf{Distance between the zeros})\label{distance between zeros}
        Let $\{X_i\}_{i=1}^\infty$ be a sequence of i.i.d. random variables with law $\mu$, supported in the some compact set $K$. If there exists a real-valued function $f$ such that $$\mathbb{P}\left( |z-X_j|>t \right)\geq 1-f(t), $$ for all $z \in K$ and $t$ small, then for any set $B\subset K$, we have
        \begin{align}
             &\mathbb{P}\left(  \min_{2\leq j \leq n}|X_1-X_j| >t  \Big| X_1 \in B\right) \geq \left(1-f(t)\right)^n.
        \end{align}
    \end{Lemma}
    \begin{proof}[ \textbf{Proof of Lemma \ref{distance between zeros}}]
    We use the independence of the random variables after conditioning on $X_1$ to write
            \begin{align*}
                \mathbb{P}\left(  \min_{2\leq j \leq n}|X_1-X_j| >t \Big|X_1 \in B\right)
                &= \frac{1}{\mathbb{P}( X_1 \in B)}\bigintsss_{K} \mathbb{P}\left(  \min_{2\leq j \leq n}|X_1-X_j| >t,X_1 \in B \Big|X_1=z \right) d\mu(z)\\[.2em]
                &=\frac{1}{\mathbb{P}( X_1 \in B)}\bigintsss_{B} \mathbb{P}\left(  \min_{2\leq j \leq n}|z-X_j| >t\right) d\mu(z)\\[.2em]
                &=\frac{1}{\mathbb{P}( X_1 \in B)}\bigintsss_{B} \mathbb{P}\left( |z-X_j|>t\right)^{(n-1)} d\mu(z)\\[.2em]
                &\geq \left(1-f(t)\right)^{(n-1)},
            \end{align*}
        where we got the last inequality from the hypothesis of the theorem.
    \end{proof}

    In the next two lemmas, $P_n(z)$ will denote a \emph{random} polynomial defined by $P_n(z) = \prod_{j=1}^{n}(z- X_j),$ where $\{X_j\}$ is a sequence of of i.i.d. random variables.

    \begin{Lemma}\textbf{(Lower bound on first derivative)}\label{derivative L bound}
          Let $L$ be a compact set with $0< c(L)<1$.  Let $\nu$ be the equilibrium measure of $L.$ Consider a sequence of i.i.d. random variables $\{X_i\}_{i=1}^\infty$  with law $\nu$. Assume that for every $1\leq p \leq 3$, there exists some positive constant $C_p> 0$, such that $\mathbb{E}\left[\left|\log|z-X_1|\right|^p\right] <C_p.$ Then for $n$ large, there exists a constant ${C}>0$, depending on $L$ such that,
          \begin{align}\label{l36}
              \mathbb{P}\left(\big|{P}^{'}_n(X_1)\big| \geq (c(L))^{2n}\right) \geq 1 -\frac{{C}}{\sqrt{n}}.
          \end{align}
    \end{Lemma}
    
    \begin{proof}[ \textbf{Proof of Lemma \ref{derivative L bound}}]
    We start by taking the logarithm to write
    \begin{align}\label{l21}
                \mathbb{P}\left(\big|{P}^{'}_n(X_1)\big| \geq (c(L))^{2n} \right)
                &= \mathbb{P}\left(\sum_{j=2}^n \log|X_1-X_j|\geq 2n\log c(L)\right)\nonumber \\
                &= \bigintss_{L} \mathbb{P}\left(\sum_{j=2}^n \log|z-X_j|\geq 2n\log c(L) \right)d\mu(z).
            \end{align}
By Frostman's theorem (c.f. \cite[Theorem-$3.3.4$]{Ransfordpotentialthryincomplexplane}), $\log|z-X_j|$ are i.i.d. random variables with mean $\log(c(L))$ for $\nu$ all most every $z \in L.$ Therefore adding and subtracting the mean in \eqref{l21} we get, 
            \begin{align}\label{l21.5}
                \bigintss_{L}\mathbb{P}\left(\sum_{j=2}^n \left(\log|z-X_j|-\mathbb{E}[\log|z-X_j|]\right)\geq 2n\log c(L)-(n-1)\log c(L)\right)d\nu(z).
            \end{align}
    We estimate the integrand in (\ref{l21.5}) by applying the Berry–Esseen Theorem \eqref{BE}, to $Y_j = \log|z-X_j|-\mathbb{E}[\log|z-X_j|]$ and arrive at
            \begin{multline}\label{l22}
                \mathbb{P}\left(\frac{1}{\sqrt{n-1}\sigma(z)}\sum_{j=2}^nY_j \geq \frac{(n+1)}{\sqrt{n-1}}\frac{\log c(L)}{\sigma(z)}\right)d\nu(z)\\
                      \geq 
                      \bigintsss_{L} \left(\Phi\left(\frac{(n+1)}{\sqrt{n-1}}\frac{\log c(L)}{\sigma(z)}\right)-\frac{C\rho(z)}{\sigma^3(z)\sqrt{n-1}}\right) d\nu(z),
            \end{multline}
            where $\sigma^2(z)=\mathbb{E}\left[\left(\log|z-X_j|\right)^2 \right]$, $\rho(z)=\mathbb{E}\left[\left|\log|z-X_j|\right|^3 \right]$, and $\Phi$ is the distribution function of the standard normal. As $n \to \infty$ we have $\frac{(n+1)}{\sqrt{n-1}}\frac{\log c(L)}{\sigma(z)} \to -\infty$ since $c(L)<1$, consequently $\Phi\left(\frac{(n+1)}{\sqrt{n-1}}\frac{\log c(L)}{\sigma(z)}\right) \to 1$. From the hypothesis, we have uniform upper and lower bounds on $\sigma^2(z)$ and $\rho(z)$ using which we can bound the integrand in (\ref{l22}) as 
            \begin{align}\label{l35}
                \Phi\left(\frac{(n+1)}{\sqrt{n-1}}\frac{\log c(L)}{\sigma(z)}\right)-\frac{C\rho(z)}{\sigma^3(z)\sqrt{n-1}}
                \geq \left(1-\frac{C}{\sqrt{n}}\right).
            \end{align}
        Putting the bound (\ref{l35}) in the R.H.S. of (\ref{l22}) we get the required probability (\ref{l36}) which is 
        \begin{align*}
            \mathbb{P}\left(|{P}^{'}_n(X_1)| \geq c(L)^{2n}\right)&=  \bigintsss_{L} \left(1-\frac{C}{\sqrt{n}}\right)d\nu(z)
            \geq 1 - \frac{C}{\sqrt{n}}.\qedhere
        \end{align*}
    \end{proof}


\begin{Lemma}\textbf{(Bound on higher derivatives)}\label{higher derivative bounds}
         Let $\{X_i\}_{i=1}^\infty$ be a sequence of i.i.d. random variables with law $\mu$, supported on the compact set $K$. Let us define the event $A:=\{|X_i-X_j| > \frac{1}{C_n} ; \forall i \neq j \}$ for some constant $C_n>0$, then
             \begin{align}
                \mathbb{E} \left [ \frac{1}{k!}\left|\frac{P^{(k)}_n(X_1)}{{P}^{'}_n(X_1)}\right| \Big | A \right] \leq \binom{n-1}{k-1}C_n^{k-1}, \quad \quad k=2,\cdots,n-1.
             \end{align}
    \end{Lemma} 
    \begin{proof}[\textbf{Proof of Lemma \ref{higher derivative bounds}}]
    We start by writing ${P_n}(z)$ as $ {P_n}(z)=(z-X_1)Q_n(z)$, where $Q_n(z):=\prod_2^n(z-X_j)$. Then, repeated differentiation yields,
    $${P}^{(k)}_n(z)=k{Q}^{(k-1)}_n(z)+(z-X_1)Q^{(k)}_n(z).$$ Putting $z=X_1$ in the equation above, we get $\frac{ {P}^{k}_n(X_1)}{{P}{'}_n(X_1)}=\frac{k{Q}^{(k-1)}_n(X_1)}{Q_n(X_1)}$. Since $X_1$ is not a root of $Q_n(z)$,  $\frac{{Q}^{(k-1)}_n(X_1)}{Q_n(X_1)}$ will have $(n-1)(n-2)...\big(n-(k-1)\big)$ many summands of the form $\left[\frac{1}{(X_1-X_2)...(X_1-X_{k})} \right]$. Here, we only care about the number of summands as $X_i$'s are i.i.d. all of the summands will have the same expected value.
        \begin{align*}
            \mathbb{E} \left [ \frac{1}{k!}\left|\frac{P^{(k)}_n(X_1)}{{P}^{'}_n(X_1)}\right|\big| A \right] 
                &\leq \frac{k(n-1)(n-2)...(n-k+1)}{k!} \mathbb{E}\left[\left|\frac{1}{(X_1-X_2)...(X_1-X_{k})}\right| \Big| A \right] \\
                &\leq \binom{n-1}{k-1} C_n^{k-1},
        \end{align*}
    where we got the last estimate using the hypothesis of the lemma. Here, it is worthwhile to mention that we can also do similar things in the case of deterministic polynomials. 
    \end{proof}

\section{Proof of Theorem \ref{capless1}}
\begin{proof}[\textbf{Proof of $(a)$}]  Assume the contrary that $M(K) = 1.$ This means we can find a sequence of polynomials $p_{n_k}\in\mathscr{P}_{n_k}(K)$ which have $\mathscr{C}_{n_k}(K)$ many components with 
    \begin{align}\label{capacityless1}
      \underset{n \to \infty}{\overline{\lim}}  \frac{\mathscr{C}_{n_k}(K)}{n_k} =1.
    \end{align}
    Let $\mu_{n_k}$ denote the empirical measure of the zeroes of $p_{n_k}$. Note that all of these measures are supported on $K.$ Hence there exists a further subsequence  $\mu_{n_{k_l}}$ of $\mu_{n_k}$ and a measure $\mu\in\mathcal{P}(K)$ such that $\mu_{n_{k_l}} \overset{\omega^{*}}{\longrightarrow } \mu $ (c.f. \cite[Theorem A.$4.2$]{Ransfordpotentialthryincomplexplane}). From the definition of capacity it follows that if $c(K) < 1,$ then any measure $\sigma\in\mathcal{P}(K)$ satisfies $I(\sigma) < 0.$ Applying this to our weak limit $\mu,$ we have $I(\mu) = \int_{K}U_{\mu}(w)d\mu(w) < 0. $ This forces $U_{\mu}(w_0) < 0$ for some $w_0\in\supp (\mu).$ By upper semicontinuity, we can deduce that $U_{\mu} < 0$ on some closed disk $\overline{B}$ centered at $w_0$ with $\mu(B) > c_1 > 0.$ Since subharmonic functions attain their supremum on compacts, we can find $c_2>0$ such that $\sup_{\overline{B}}U_{\mu}(z)\leq - c_2 < 0.$ We now make two observations.
    
    \begin{enumerate}[(i)]
    \item By the Portmanteau Theorem (c.f. \cite[Theorem $2.1$]{Bilingsleyconvergenceofprobabilitymeasures}) we have 
    \begin{align*}
       \underset{l \to \infty}{\underline{\lim}} \mu_{n_{k_{l}}}(B) \geq \mu(B)>c_1.
    \end{align*}

    \item Lemma \ref{Main lemma} implies that 
    \begin{align*}
        \underset{l \to \infty}{\overline{\lim}} \sup_{\overline{B}}  U_{\mu_{{n_{k_l}}}} \leq    \sup_{\overline{B}} U_{\mu}\leq - c_2.
    \end{align*}
    \end{enumerate}
Combining the above two observations and remembering that $U_{\mu_n} = \frac{1}{n}\log|p_n|,$ we obtain that for all large enough $l$, $p_{n_{k_l}}$ has a single component (containing the disk $B$) which encloses at least $c_1 n$ many zeros. This contradicts \eqref{capacityless1} and hence finishes the proof. 
\end{proof}

\vspace{0.1in}

\begin{proof} [\textbf{Proof of ($b$)}]
    The idea of the proof is to construct a sequence of monic random polynomials $R_n$ with all zeros in $K$ which satisfy
    \begin{align*}
        \underset{n \to \infty}{\underline{\lim}} \mb{E}\left[ \frac{\mathcal{C}\left(\Lambda_{R_n}\right)}{n}\right] >0.
    \end{align*}
     Once this is done, we can finish the proof by the \emph{probabilistic method}. If a non-negative random variable $X$ has mean $\mb{E}(X) > 0,$ then $X\geq \mb{E}(X)$ with positive probability. This fact applied to $\frac{\mathcal{C}(\Lambda_{R_n})}{n}$ then gives the conclusion that $m(K) > 0$. We start by letting $d = \diam(K)$ be the diameter of the set $K$. Then by a well known estimate, see \cite[Theorem $5.3.4$]{Ransfordpotentialthryincomplexplane}, we have $d \geq 2c(K) >1$. Choose $\epsilon >0 $ such that $d>1+3\epsilon$. Let $a,b$ be points on $\partial K$ with $|a-b|=d$. Consider the ball $B(a,1+2\epsilon)$. This ball will certainly not cover the whole of $K$. Therefore by the Jordan curve theorem (c.f. \cite{jctfirstproof, jordancurvetheorem}) $({{B(a,1+2\epsilon)}}^c \cap K)^{\mathrm{o}}$ is a non empty open set. Hence we can fit a small $C^2$ arc $L$ inside this open set such that $b\in L.$ Let $\mu_L$ be the equilibrium measure of $L$, and  $c(L)=\delta >0$ be the capacity of $L$. Now, define the sequence of random polynomials
    \begin{align}
        R_n(z)= (z-a)^{n_1} \prod_{j=1}^{n_2}(z-X_j)=(z-a)^{n_1}P_{n_2}(z),
    \end{align}
    where $\{X_j\} $ is a sequence of i.i.d. random variables with law  $\mu_L,$ and $n_1=c_1n, n_2=c_2n$ with $c_1+c_2=1$ to be chosen later. Notice that $P_{n_2}$ is the random part in $R_n$, so we will only focus on bounding $|P_{n_2}|$ from below using ideas from \cite{ghosh2023number}. Let $Q_n(z):=\prod_{i=1}^n  (z-z_i)$ be a deterministic polynomial; then we say that a root $z_j$ forms an \emph{isolated component} if there exists a ball $\mathcal{B}$ containing $z_j$ such that, 
            \begin{align}\label{isolated}
                 \begin{cases}
                    & z_k \notin \mathcal{B}, \hspace{.57in} \forall k \neq j \\[10pt]
                    &|Q_n(z)| \geq 1,  \hspace{.25in}  \forall z \in\partial\mathcal{B}.
                 \end{cases}
            \end{align}
     We now define for each $1\leq i\leq n,$ the event $T_i=\{X_i\hspace{0.05in}\mbox{forms an isolated component}\}$. Then it follows that 
    \begin{align*}
         \mb{E}\left[ \frac{C\left(\Lambda(R_n)\right)}{n}\right] \geq  \mb{E}\left[ \frac{1}{n}\sum_1^{n_2}  \mathbbm{1}_{T_i}\right] = c_2 \mb{P}(T_1).
    \end{align*}
    We will be done if we can show that $\mb{P}(T_1) > c>0$ for some $c$ independent of $n.$ Towards this, we start by providing a sufficient condition for having an \emph{isolated component} at $X_1$.
    Let $z \in B(X_1,r_{n_2})$, with $r_{n_2}:= \frac{1}{n_2^6}$. Near $X_1$ we expand the random part of the polynomial in the Taylor series to obtain,
    \begin{align}\label{sufficient condition for component}
        |R_n(z)| &= |z-a|^{n_1} \left| P_{n_2}(X_1)+\cdots+ P_{n_2}^{(k)}(X_1)\frac{r_{n_2}^k}{k!}+\cdots+P^{({n_2})}_{n_2}(X_1)\frac{r_{n_2}^{n_2}}{{n_2}!}\right|\nonumber \\
        & \geq (1+\epsilon)^{n_1} \left|P^{'}_{n_2}(X_1)r_{n_2}\right| \left|1- \sum_{k=2}^{n_2}\frac{|P^{({k})}_{n_2}(X_1)\frac{r_{n_2}^{k}}{{k}!}|}{|P^{'}_{n_2}(X_1)\frac{r_{n_2}}{1!}|}\right|.
    \end{align}
    To get an \emph{isolated component}, the quantity in \eqref{sufficient condition for component} needs to be greater than $1$. To guarantee this, we define events $F_1,...F_{{n_2}+1}$ in \eqref{1} as follows.
        \begin{align}\label{1}
           \begin{cases}
                &F_1:=\left \{ \big|{P}^{'}_{n_2}(X_1)\big| \geq c(L)^{2n_2} \right \},\\[.8em]
            &F_k:=\left \{ \Big|\frac{P^{(k)}_{n_2}X_1)\frac{r_{n_2}^k}{k!}}{{P}^{'}_{n_2}(X_1)\frac{r_{n_2}}{1!}}\Big| < \frac{1}{2n_2^2}  \right\}, \quad \textit{for } k =2,...,{n_2},\\[1.2em]
            &F_{{n_2}+1}:=\left \{ \underset{2\leq j \leq {n_2}}{\min}|X_1-X_j| > \frac{1}{n_2^6}\right \},
           \end{cases}
        \end{align}
        with $r_{n_2} = \frac{1}{n_2^6}$. Notice that on these events, we have for all $z\in\partial B(X_1, r_{n_2})$
        \begin{align}\label{choosing c_1,c_2}
            |R_n(z)| \geq (1+\epsilon)^{n_1} c(L)^{2n_2} \frac{1}{2n_2^6} \geq \big((1+\epsilon)^{c_1}c(L)^{c_2}\big)^n \frac{1}{2n_2^6}.
        \end{align}
        Since $\epsilon$ and $c(L)$ are fixed, we can choose $c_1,c_2 > 0$ in such a way that $\left((1+\epsilon)^{c_1}c(L)^{c_2}\right) >1$. This choice substituted back into \eqref{choosing c_1,c_2} ensures a isolated component near $X_1$. Now we use the assumption that the set $L$ is $C^2$ to claim that two zeros can not be close together with high probability. Fix  a root say $X_1,$ and define the event $A:=\left \{ \underset{2\leq j \leq {n_2}}{\min}|X_1-X_j| >\frac{1}{n_2^6}\right \}$. Then from Lemma \ref{distance between zeros} we obtain $\mb{P}(A) \geq \left(1-f\left(\frac{1}{n_2^6}\right)\right)^{(n-1)}$. On the other hand, Lemma \ref{Bound on probability for equilibrium measures} implies that for $C^2$ Jordan arcs we can take $f(t)=C\sqrt{t}$. Plugging this into the estimate for $\mb{P}(A)$, we get the lower bound   
        \begin{align*}
            \mb{P}(A) \geq 1- \frac{C}{\sqrt{{n_2}}}.
        \end{align*}
       We can now get a lower bound on $\mb{P}(T_1)$ in the following way. 
    \begin{align}\label{final1}
         \mb{P}(T_1)\geq\mb{P}\left( T_1 \cap A \right)= \mb{P}\left( T_1 | A \right)\mb{P}(A)\geq \mb{P}\left( T_1 | A \right)- \frac{C}{\sqrt{{n_2}}} \geq \mb{P}\left( \cap_{j=1}^{n_2+1}F_j |A\right) - \frac{C}{\sqrt{{n_2}}}.
    \end{align} 
    We estimate the conditional probability of  $F_1$ given $A$  directly by Lemma \ref{derivative L bound}, where the hypotheses of log moments bounds are satisfied by Lemma \ref{log moment bound for uniform on cicle} to obtain
      \begin{align}\label{lo2}
          \mb{P}(F_1|A) \geq 1- \frac{C}{\sqrt{{n_2}}}.
      \end{align}
      Bounds for events
      $F_k$ for $ k=2,\cdots,n_2$ follows directly from Lemma \ref{higher derivative bounds} with $C_n=1/n_2^6$ and using Markov inequality in the following manner 
      \begin{align} \label{lo1}
\mb{P}\left(\Big|\frac{P^{(k)}_{n_2}X_1)\frac{r_{n_2}^k}{k!}}{{P}^{'}_{n_2}(X_1)\frac{r_{n_2}}{1!}}\Big| \geq \frac{1}{2n_2^2} \Big |A \right)\leq 2n_2^2 \binom{{n_2}-1}{k-1} \left(\frac{1}{n_2^6}\right)^{k-1} \leq \frac{1}{n_2^2}.
      \end{align} 
      Finally notice that $F_{n_2+1} \subset A ,$ therefore  $\mb{P}(F_{n_2+1}|A) =1$. Subbing the estimates \eqref{lo2} and \eqref{lo1} together into \eqref{final1}, we get 
      \begin{align*}
           \mb{P}(T_1) \geq \mb{P}(F_1|A)-\sum_{i=2}^{n_2+1}\mb{P}(F_1|A)^c-\frac{C}{\sqrt{{n_2}}} \geq 1- \frac{C}{\sqrt{{n_2}}}.
      \end{align*} 
    Here $C$ is independent of ${n_2}$. Therefore, we get the desired conclusion for all $n$ large enough.
\end{proof}

\begin{proof}[\textbf{Proof of $(c)$}]
Let $K$ be connected and $c(K)< \frac{1}{4}.$ Since the diameter $d$ of $K$ satisfies $c(K)\geq \frac{d}{4}$, see for instance  \cite{Ransfordpotentialthryincomplexplane},  it follows that $d<1.$ This implies that for any polynomial $p(z) = \prod_{j=1}^{n}(z- a_j)\in\mathscr{P}_n(K), $ we have $|p(z)|< 1$ for all $z\in K$. Hence $K\subset\Lambda_p.$ But since the zeros themselves are contained in $K$ which is assumed connected,  $\Lambda_p$ is connected. Therefore $\mathscr{C}_n(K) = 1$ for all $n,$ which in particular gives $m(K) = 0.$ Suppose $c(K) =\frac{1}{4}$, then the diameter $d$ of $K$ is less than equal to $1$. The case $d<1$ is the same as before, so we consider  $d=1$. In this case, for any polynomial $p \in \mathscr{P}_n(K)$ we have $K \subset \overline{\Lambda_p}$. That is, $\overline{\Lambda_p}$ has exactly one component. If $\Lambda_p$ has more than one component then, by Lemma \ref{touching of lemniscates} there exists a critical point $w$ of $p$ such that $|p(w)|=1.$ Since the diameter is $1$, we have 
\begin{align*}
    |w-z_j|=1 \quad \forall j \in \{1,...,n\}.
\end{align*}
where $z_1,...,z_n$ are the zeros of $p$. By translation, we can assume $w=0$. Now choose any zero say $z_1,$ and consider the ball $B(z_1,1)$. As $0$ is a critical point, $z_j$ for $j\neq 1$ cannot lie in ${B(z_1,1) \cap \partial \mb{D}}$ by the Gauss-Lucas theorem. On the other hand if some $z_k$ is outside $B(z_1,1)$ then $d \geq |z_1-z_k|>1$, which leads to a contradiction. Therefore, $\Lambda_p$ has only one component i.e. $\mathscr{C}_n(K) = 1$, and $m(K)=M(K)=0$.
\end{proof}

\section{Proof of Theorem \ref{capbigg1}}

\begin{proof}[\textbf{Proof of $(a)$}]
  Our proof rests on the \emph{probabilistic method} and the following theorem of Krishnapur, Lundberg, and Ramachandran, see \cite{KLR}.
  
\begin{thm}\label{KLR 23}
  Let $K$ be a compact set with $c(K) > 1.$ Suppose that the equilibrium measure $\nu$ of $K$ satisfies
  \begin{align}
        \nu (B(z, r)) \leq  Cr^{\epsilon}.
    \end{align}
 Consider the sequence of random polynomials defined by $P_n(z):=\prod_1^n(z-X_i)$, where $\{X_i\}^\infty_1$ is a sequence of i.i.d. random variables with law $\nu.$ Then there exist constants $c_0, c_1 > 0$ such that,

\begin{equation}\label{smalllemniscate}
\Lambda_{P_n}\subset \cup_{k=1}^n B(X_k, e^{-c_0n}),    
\end{equation}

holds with probability at least $1-  e^{-c_1n}$ for all large $n.$
   
\end{thm}
Following the notation used in the above theorem, we claim that there exist constants $c, \alpha > 0$ such that for each $n,$ the event
\[B_n = \left\{\min_{1\leq i < j\leq n}|X_i - X_j|> \frac{c}{n^{\alpha}}\right\}\]
satisfies $\mathbb{P}(B_n)\geq\frac{1}{2}$.   Assuming the claim for now, let us finish the proof of part $(a).$ For $n\in\N,$ define the event $A_n = \left\{\Lambda_{P_n}\subset \cup_{k=1}^n B(X_k, e^{-c_0n})  \}\right\}$. Then by \eqref{smalllemniscate} we have that $\mathbb{P}(A_n)\geq 1-  e^{-c_1n} $ holds for all large $n$. Therefore $\mathbb{P}(A_n\cap B_n)\geq\frac{1}{4}$ for large $n.$ This implies in particular that for all $n$ large, there exists some choice of points $w_{1, n}, w_{2, n},... w_{n,n}$ in $K$ such that the polynomial $p_n(z) = \prod_{k=1}^{n}(z-w_{k, n})$ has $n$ isolated components for $\Lambda_{p_n}$ (since distance between the zeros is at least polynomially decaying in $n,$ while the balls are all of exponentially small radius). It remains to prove the claim. 

\vspace{0.1in}
For the proof of the claim, we first observe that if $z\in K,$ and $t>0,$ we have $\PP(|z - X_j|< t) = \nu(B(z,t)\leq Ct^{\epsilon}$ by hypothesis on $\nu.$ Hence if $i, j\in [n]$ with $i\neq j$, independence of the random variables and simple conditioning yields 
\begin{equation}\label{smalldist}
\PP(|X_i - X_j|< t) = \int_{K}\PP(|z - X_j|< t)d\nu(z)\leq Ct^{\epsilon}.    
\end{equation}
The estimate \eqref{smalldist} along with the union bound gives 
\begin{align*}\label{mindist}
\PP\left(\min_{1\leq i < j\leq n}|X_i - X_j|\geq t \right)&= 1-\sum_{i <j}\PP(|X_i - X_j| < t)\\
&\geq 1 - n^2 C t^{\epsilon}\geq\frac{1}{2}
\end{align*}

if $t = \frac{1}{(2Cn^2)^{\frac{1}{\epsilon}}}.$ This finishes the proof of the claim and hence of part $(a).$
\end{proof}

\begin{proof}[\textbf{Proof of $(b)$}]  This proof must be well known but we couldn't find a suitable reference. We start by recalling that a Fekete $n$-tuple of $K$ is any $n$-tuple $(z_1, z_2, ..., z_n)$ which realizes the supremum

\[\sup_{j, k\leq n : j < k}\left\{\prod |w_j - w_k|^{\frac{2}{n(n-1)}}: w_1, w_2, ... w_n\in K\right\}.\]

For $n \in \mb{N}$, let $\{z_{k,n}\}_{k=1}^{n}$ denote a Fekete $n$-tuple for the compact set $K$. Let $p_n$ denote the corresponding Fekete polynomial of degree $n$ defined by
    \begin{align*}
        p_n(z) = \prod_{k=1}^n (z-z_{k,n}), \quad \quad \forall n \in \mb{N}.
    \end{align*}
    We will prove that the  $\{|p_n| < 1\}$  has $n$ connected components for large enough $n$. The idea is to use Lemma \ref{suffcondn} to show that at each zero $z_{k, n}$ there is an \emph{isolated component} of the lemniscate $\Lambda_{p_n}$. We will first check that $\Big|p^{'}_n(z_{1,n})\Big|$ is exponentially large using properties of the Fekete points. Similar estimates obviously hold at the other zeros. Let $l_{r,n}$ denote the $r^{\textit{th}}$ incomplete polynomial or order $n$ which are defined by
    \begin{align*}
        l_{r,n}(z) := \prod_{i \neq r} \frac{(z-z_{i,n})}{(z_{r,n}-z_{i,n})}.
    \end{align*}
    We know from a property of the Fekete points (\cite[Exercise-$5.5.3$]{Ransfordpotentialthryincomplexplane}) that  $||l_{r,n}||_K \leq 1$. Using this we estimate
    \begin{align}\label{fekete derivative estimate}
        |p_n^{'}(z_{1,n})| =\prod_{i \neq 1} \left| z_{1,n}-z_{i,n}\right| \geq \sup_{z \in K} \left(\prod_{i \neq 1}\left| z-z_{i,n} \right| \right) \geq c(K)^{n-1}, 
    \end{align}
    where the rightmost inequality follows from the fact that the sup norm of a monic polynomial of degree $d$ on a compact $K$ is at least $c(K)^{d}$. Since $c(K)>1$, we get from \eqref{fekete derivative estimate} that $\Big|p^{'}_n(z_{1,n})\Big|$ is exponentially large. Next, a result of Kovari and Pommerenke, c.f., \cite{Kovaripommerenke}, implies that the minimum distance between \emph{any two} Fekete points of a continuum is at least of the order $\dfrac{1}{n^2}$. Hence Lemma \ref{suffcondn} applies here to show the existence of $n$ components.  
\end{proof}

\section{Proof of Proposition \ref{proposition} and Theorem \ref{capeq1M}}

\begin{proof}[\textbf{Proof of proposition \ref{proposition}}]
Let us assume first that $K$ is a \emph{period-$m$ set} with $P_m$ as \emph{generating polynomial}. Let $\mathcal{T}_n(x)$ be the Chebyshev polynomial of degree $n$. We define polynomials $Q_{nm}(z):=\mathcal{T}_n(P_m(z))$ of degree $mn$. Notice that, by construction $Q_{nm}(z) \in \mathcal{P}_n(K)$. Taking the derivative we see that $Q_{nm}^{'}(z)=\mathcal{T}_n^{'}(P_m(z))P_m^{'}(z)$. There are two kinds of critical points. The first kind is the critical points of $P_m$, which we ignore. The second kind is the inverse images of critical points of $\mathcal{T}_n(x)$ under $P_m$, and the critical points of $P_m$, that is $\{z: P_m^{-1}(z)=cos\left(\frac{2\pi}{n}\right), k=1,\cdots,n-1\}$. For all these $(n-1)m$ critical points of $Q_{nm}$, the critical values are $2$. Therefore, by lemma \ref{Components and critical points}, the lemniscate of $Q_{nm}$ has at least $(n-1)m$ many components. Now fixing $m$ and taking limit $n \to \infty$ we get,
   \begin{align}\label{M1}
      \underset{n \to \infty}{\overline{\lim}}   \frac{\mathscr{C}_n(K)}{n} \quad \geq  \quad \underset{n \to \infty}  {\overline{\lim}}  \frac{(n-1)m}{mn} \quad = \quad 1.
    \end{align} 
If $K$ is a \emph{closed lemniscate}, similar arguments show \eqref{M1}. The only modification is that we need to compose the \emph{generating lemniscate} with $z^n+1$ instead of $\mathcal{T}_n$. We proceed to the proof of the stronger result that $\mathscr{C}_n(K)= n$ along a subsequence. Let $Q_m$ be the \emph{generating polynomial} of the lemniscate. We define a sequence of polynomials by
    \begin{align}\label{polyseqforlemniscate}
        P_{nm}(z):= \left(Q_m(z)\right)^n+1,\hspace{0.1in}n\in\N.
    \end{align}
There are two cases. Assume first that $K$ has $m$ components. Then all the critical values of $Q_m$ are greater than $1$ in absolute value. Therefore, for large enough $n$ all the critical values of $P_{nm}, $ would be greater than equal to $1$ in absolute value. By lemma \ref{Components and critical points}, we get $\mathscr{C}_{mn}(K)=mn$ for all $n$ large. Next, assume that $K$ has strictly less than $m$ components. We then order the $(m-1)$ critical values of $Q_m$ in the following way $r_1e^{i\theta_1},\cdots,r_{m-1}e^{i\theta_{m-1}}$, where $r_1 \leq r_2 \leq \cdots \leq r_{m-1}$, and $r_1 \leq 1$. Now consider the sequence of $(m-1)$ tuples $\{(n\theta_1,...,n\theta_{m-1})\}_{n=1}^{\infty}$ mod $2\pi \in [0,2\pi]^{m-1}.$ Given a small $\delta$, we can divide $[0,2\pi]^{m-1}$ into a finite number of small cubes of diameter less than $\delta$. Then, by the pigeonhole principle, there is a subsequence $\mathcal{N} \subset \mb{N}$ for which all of the $(m-1)$ tuples are in the same cube. Now, subtracting the first element of $\mathcal{N}$ from all the other elements, we get a subsequence $\mathcal{N}'$. For any $n_k \in \mathcal{N}'$ the $(m-1)$ tuple $(n_k\theta_1,...,n_k\theta_{m-1})$ mod $2\pi$ is at most $\delta$ distance away from $(0,...,0)$. Therefore we have that 
\begin{align*}
    \Re \left({r_l}^{n_k} e^{in_k\theta_l}\right) \geq 0, \quad  \quad \forall \hspace{.05in}1 \leq l \leq (m-1), k \in \mb{N}.
\end{align*}
This shows along the subsequence $\mathcal{N}'$ all the critical values of the polynomials in \eqref{polyseqforlemniscate} will be greater than or equal to $1$ in modulus, thus concluding the proof.
\end{proof}


\begin{proof}[\textbf{Proof of Theorem \ref{capeq1M}} ]
    We will first prove that $M(K) = 1.$ We make use of the following lemma. 
    
    \begin{Lemma}\label{approximation of lemniscate lemma}
    Let the compact set $K$ be the closure of a bounded Jordan domain with $C^{2}$ boundary. Further, assume that $c(K)=1$. Then, given any $\delta \in (\frac{1}{2},1)$ there exists a monic polynomial $F_n$ of degree $n = n(\delta)$ such that,
    \begin{align}\label{approximating poly equation}
        F_n^{-1}\left(\overline{B(0,\delta)}\right) \subset K^{o},
    \end{align}
where $K^{o}$ denotes the interior of $K.$
\end{Lemma}

\begin{proof}[Proof of Lemma \ref{approximation of lemniscate lemma}]
    We will show that the Faber polynomials (defined below) satisfy the criterion.
    Let $\Phi$ be the exterior conformal map from $K^c \cup \infty$ onto the exterior of the unit disc satisfying $\Phi(\infty) = \infty$ and $\Phi^{'}(\infty)>0$. Since the capacity of $K$ is $1$, we can write $\Phi(z)= z+c_0+\frac{c_{-1}}{z}+\cdots $, in a neighbourhood of $\infty$. Since $\partial K$ is $C^2,$ the map $\Phi$ extends to a homeomorphism of the corresponding boundaries. The Faber polynomials $F_n$ of degree $n$ are defined to be the polynomial part of $\Phi^{n}$, that is 
    \begin{align*}
        \Phi^{n}= \left(z+c_0+\frac{c_{-1}}{z}+\cdots\right)^n= \underbrace{z^n+nc_0z^{n-1}+\cdots }_{F_n}+\mathcal{O}\left( \frac{1}{z} \right).
    \end{align*}
    \noindent  When the boundary is $C^2$, we have the following estimate, c.f. \cite[Theorem-$1.1$]{seriesinfaberpolynomialSuetin}
    \begin{align}\label{faber poly estimate}
        |F_n(z)-\Phi^n(z)| \leq C_0 \frac{\log n}{n},
    \end{align}
    uniformly for $z \in K^c\cup\partial K$. For our purpose this estimate is sufficient. We first show that all the zeros of $F_n$ (for all $n$ large) are strictly inside $K$.
    Let $z_0$ be a zero of $F_n$ that is outside $K$, then from \eqref{faber poly estimate} we have for $n$ large
    \begin{align*}
        |\Phi^n(z_0)| =  |F_n(z_0)-\Phi^n(z_0)| \leq C_0 \frac{\log n}{n} <1.
    \end{align*}
    This contradicts the fact that $\Phi$ maps the exterior of $K$ to the exterior of the disk. To show the condition in \eqref{approximating poly equation}, it is enough to show that for some $n$ large,
     \begin{align}\label{delta condition}
        \inf\{ |F_n(z)| : z \in \partial K\} > \delta.
    \end{align}
    Indeed, suppose that \eqref{delta condition} is satisfied and some component of $F_n^{-1}\left(\overline{B(0,\delta)}\right)$ goes outside of $K$. Then one of the following two cases occurs. Either this component is disjoint from $K$ in which case it will contain a zero of $F_n$. This contradicts the fact that all zeros of $F_n$ are strictly inside $K$. The alternative is that this component intersects $K$ non trivially. But then the estimate \eqref{delta condition}  will be violated. This proves that $F_n^{-1}\left(\overline{B(0,\delta)}\right)\subset K^{o}.$  Finally, the estimate \eqref{delta condition} also follows from \eqref{faber poly estimate} in the following manner.
    \begin{align*}
        |F_n(z)|\geq \left|\Phi^n(z)\right|-|F_n(z)-\Phi^n(z)| \geq 1- C_0 \frac{\log n}{n} > \delta,
    \end{align*}
    for all $z \in \partial K$ and $n$ large enough. This proves the lemma.
\end{proof}

With the lemma proved, we continue with the proof of the Theorem. Our proof crucially uses the following polynomial $E_n$ first considered by Erd\"os, Herzog, and Piranian.
    \begin{align}\label{EHP polynomials}
        E_n(z):= \frac{(z^n+1)(z-1)^2}{(z-e^{i\pi/n})(z-e^{-i\pi/n})}.
    \end{align}
    Note that $E_n$ is obtained from the polynomial $z^n +1$ by replacing the two zeros closest to $z=1$ by two zeros at $z=1.$ In particular, all the zeros of $E_n$ are on the unit circle and since it is a double zero, $z=1$ is a critical point of $E_n.$ Erd\"os et. al. in in \cite[Theorem $7$]{EHP} showed that for all $n$ large, $\overline{\Lambda_{E_n}}$ has exactly $(n-1)$ components. Hence, $\Lambda_{E_n}$ also has exactly $(n-1)$ components.
    Let $\{\beta_{n,j}\}_{j=1}^{n-1}$ be the critical points of $E_n$, labelled so that $\beta_{n,1}=1$ for all $n$. By Lemma \ref{Components and critical points}, we have that all critical values are greater than equal to $1$ except for $\beta_{n,1}=1$ where the critical value is $0$. But Lemma \ref{touching of lemniscates} implies that no critical value can be $1$ in modulus, therefore 
    \begin{align}\label{lower bound c_n}
        c_n =\min \Big\{|E_n(\beta_{n,j})| : {j \in \{2,\cdots,n\}} \Big\}>1 \quad \quad \forall n \in \mb{N}.
    \end{align}
    Since the $n$th roots of $-1$ are dense on the unit circle, some zero of $E_n$ lies near $-1$. By the truth of Sendov's conjecture (c.f. \cite{sendovconjecturetao}), this in turn implies that there is always a critical point of $E_n$ in $B(-1,1) \cap \mb{D}$ for all $n$ large. Denote this critical point by $\beta_{n,j_0}$. Since $\beta_{n,j_0}$ lies away from $1$, a crude bounds yields 
    \begin{align}\label{upper bound c_n}
        c_n \leq |E_n(\beta_{n,j_0})| \leq 32.
    \end{align}
 From the estimates in \eqref{lower bound c_n} and \eqref{upper bound c_n} it follows that,
 \begin{align*}
     \delta_n := \left(\frac{1}{c_n}\right)^{\frac{1}{n}} <1 \quad  \mbox{and} \quad  \lim_{n \to \infty } \delta_n = 1.
 \end{align*}
 Fixing this sequence $\delta_n,$ we define the $\delta_n$-scaled $E_n$ polynomials as
 \begin{align}
     E_{\delta_n,n}(z) := \delta_n^{n}E_n\left(\frac{z}{\delta_n} \right) = \frac{(z^n+\delta_n^n)(z-\delta_n)^2}{(z-\delta_ne^{i\pi/n})(z-\delta_ne^{-i\pi/n})}.
 \end{align}
 All the zeros of the polynomial $E_{\delta_n,n}$ lie on $\partial B(0,\delta_n),$ and using Proposition \ref{Scaling of zeros and corresponding critical points}, it is easy to see that one critical value is $0$. All the other critical values are greater than or equal to $1$. Therefore, by Lemma \ref{Components and critical points} we have that $\Lambda_{E_{\delta_n,n}}$  has exactly $(n-1)$ components. Let us choose $n_0$ large, such that $\delta_{n_0}$ is close to $1$. From Lemma \ref{approximation of lemniscate lemma}, we get a monic polynomial $F_n$ such that \eqref{approximating poly equation} holds with $\delta_{n_0}$, where $n$ depends on $\delta_{n_0}$. Now consider the polynomial defined by
 \begin{align*}
     Q(z)=E_{\delta_{n_0},n_0}\left(F_n(z)\right).
 \end{align*}
 We claim that $Q\in\mathscr{P}_{nn_0}(K).$ It is clear that $Q$ is monic of degree $nn_0.$ We have to show that all the zeros of $Q$ are inside $K$. Indeed, let $Q(w)=0.$ Then $|F_n(w)|=\delta_{n_0},$ which in turn implies by \eqref{approximating poly equation} that $w \in K$. This proves the claim. We will next show that $\Lambda_Q$ has a rather large number of connected components. Differentiating $Q$ we get, 
 \begin{align}\label{derivQ}
      Q^{'}(z)=E^{'}_{\delta_{n_0},n_0}\left(F_n(z)\right)F^{'}_n(z).
 \end{align}
 We observe from \eqref{derivQ} that there are three kinds of critical points of $Q$. The first kind are critical points of $F_n(z)$, the second kind are $F_n^{-1}(1)$ both of which we discard. The third kind which we will keep is the set $\mathcal{W}:=\big\{z : F_n(z)= \delta_{n_0} \beta_{n_0,j} , j \in \{2,\cdots,n\}\big\}$. Notice that $|\mathcal{W}|=n(n_0-2).$ For $z \in \mathcal{W},$ the corresponding critical value is
 \begin{align}\label{local001}
    |Q(z)|=|E_{\delta_{n_0},n_0}\left(F_n(z)\right)|=|E_{\delta_{n_0},n_0}\left(\delta_{n_0}\beta_{n_0,j}\right)| \geq 1.
 \end{align}
 Using \eqref{local001} with Lemma \ref{Components and critical points}, we see $\Lambda(Q)$ has at least $n(n_0-2)$ components. We can repeat this procedure for larger and larger $n_0$, and taking the limit we obtain,
 \begin{align*}
      \underset{n \to \infty}{\overline{\lim}}   \frac{\mathscr{C}_n(K)}{n} \geq \underset{n_0 \to \infty}{\overline{\lim}} \frac{n(n_0-2)}{nn_0} =1.  
 \end{align*}
This proves that $M(K) = 1.$ 

\vspace{0.1in}

\noindent We will next show that $m(K)\geq\frac{1}{2}.$ The proof is inspired by ideas from \cite{ghosh2023number} with some simple modification. Let $\nu$ denote the equilibrium measure of $K.$ Consider a sequence of random polynomials $P_n(z) = \prod_{j=1}^{n}(z- X_j),$ where $\{X_j\}$ are i.i.d. random variables with law $\nu$. Our strategy will be to show that each $X_j$ where $1\leq j\leq n,$ forms an \emph{isolated component} with probability nearly $\frac{1}{2}$. Once proven, this last fact will yield $\underset{n \to \infty}{\underline{\lim}}\mb{E}\left(\dfrac{\mathcal{C}_n}{n}\right)\geq\frac{1}{2}.$ As in the proof of \ref{capless1} $(b)$, we can then conclude that $m(K)\geq\frac{1}{2}$. Here, to keep the notation simple, we have denoted $\mathcal{C}_n:=\mathcal{C}(\Lambda_{P_n})$. We will once again use Lemma \ref{suffcondn} to show the presence of isolated components. 

\textbf{Step $1$:} We show that for small $\epsilon> 0,$ one has $|P_n^{'} (X_1)|\geq \exp(n^{\frac{1}{2} - \epsilon})$ with probability nearly $\frac{1}{2}.$ Indeed,
\begin{align*}
\mb{P}\left(\log|P_n^{'}(X_1)|\geq n^{\frac{1}{2} - \epsilon} \right)& = \int_{K} \mb{P}\left(\sum_{j=2}^{n}\log|z - X_j|\geq n^{\frac{1}{2} - \epsilon}\right)d\nu(z)\\
&= \int_{K} \mb{P}\left(\frac{1}{\sqrt{n}}\sum_{j=2}^{n}\log|z - X_j|\geq n^{ - \epsilon}\right)d\nu(z)
\end{align*}
Since $c(K) = 1,$  we have $\mb{E}(\log|z- X_1|) = U_{\mu}(z) = 0$ by Frostman's Theorem.  Note also that since $\partial K$ is $C^2$ smooth, the equilibrium measure $\nu$ is sufficiently regular by Lemma \ref{Bound on probability for equilibrium measures} so that $\log|z- X_1|$ posseses moments of order upto four, which are uniformly bounded for $z\in K.$ Now we can use Berry-Esseen \ref{BE} as before and show that 
$\mb{P}\left(\log|P_n^{'}(X_1)|\geq n^{\frac{1}{2} - \epsilon}\right)\geq \frac{1}{2} - \frac{C}{n^{\epsilon}}.$

\vspace{0.1in}

\textbf{Step $2$:} The second step is to note that $\mb{P}\left(\min_{i\neq j}|X_i - X_j|\geq\frac{1}{n^3} \right)\geq 1 - \frac{C}{n}.$ This proof is also a consequence of Lemma \ref{Bound on probability for equilibrium measures} and the reader is referred to earlier arguments.

These two steps guarantee by Lemma \ref{suffcondn} that $X_1$ forms an isolated component with probability at least $\frac{1}{2} - \frac{C}{n^{\epsilon}}.$ Hence $\mb{E}(\mathcal{C}_n)\geq \dfrac{n}{2} - n^{1-\epsilon}.$ Since $\epsilon > 0$ was arbitrary, this shows $\underset{n \to \infty}{\underline{\lim}}\mb{E}\left(\dfrac{\mathcal{C}_n}{n}\right)\geq\frac{1}{2}$. Hence the proof.

\end{proof}

%




\bibliographystyle{siam}
\bibliography{ref}

\end{document}